\documentclass[12pt]{amsart}

\usepackage{amsmath, amssymb}

\usepackage[dvipdfmx]{graphics}
\usepackage{here}
\usepackage{tikz}
\usepackage{tikz-cd}
\usetikzlibrary{arrows.meta}

\usepackage{amsmath, amssymb}
\usepackage{pifont}
\usepackage{booktabs}

\numberwithin{equation}{section}
\newtheorem{theorem}{Theorem}[section]  
\newtheorem{theorem?}{``Theorem''}[section]  
\newtheorem{corollary}[theorem]{Corollary}
\newtheorem{proposition}[theorem]{Proposition}
\newtheorem{lemma}[theorem]{Lemma}

\theoremstyle{definition}
\newtheorem{definition}[theorem]{Definition}

\theoremstyle{remark}
\newtheorem{remark}[theorem]{Remark}

\newcommand{\R}{{\mathbb R}}
\newcommand{\C}{{\mathbb C}}
\newcommand{\PP}{{\mathbb P}}

\newcommand{\N}{{\mathbb N}}
\newcommand{\Z}{{\mathbb Z}}

\makeindex             

\begin{document}

\title[On horn-shaped Reinhardt domains]{
On horn-shaped Reinhardt domains}
\author{
Chika Hayashida, 
Joe Kamimoto and Ninh Van Thu}

\address{Faculty of Mathematics, 
Kyushu university, 
Motooka 744, Nishi-ku, 
Fukuoka, 819-0395, Japan.}  
\email{hayashida.chika.311@s.kyushu-u.ac.jp}

\address{Faculty of Mathematics, 
Kyushu university, 
Motooka 744, Nishi-ku, 
Fukuoka, 819-0395, Japan.}  
\email{joe@math.kyushu-u.ac.jp}

\address{Faculty of Mathematics and Informatics, 
Hanoi University of Science and Technology, No. 1 
Dai Co Viet, Hai Ba Trung, Hanoi, Vietnam}
\email{thu.ninhvan@hust.edu.vn}
%
%
\keywords{horn domain, 
Bergman space, 
Bergman kernel, 
Bergman metric, 
holomorphic sectional curvature, 
holomorphic equivalence problem}
\subjclass[2020]{32A36 (32A40, 53C55, 32M05).}

\begin{abstract}
In this paper, 
we introduce a new class of  
unbounded Reinhardt domains in $\C^n$ ($n\geq 2$), 
which we call {\it horn domains}. 
We study 
the dimensions of the Bergman spaces, 
the holomorphic sectional curvature 
of the Bergman metric,
and the holomorphic equivalence problem
for such horn Reinhardt domains. 
\end{abstract}

\maketitle

\section{Introduction}

In this paper, 
we introduce a new class of unbounded 
Reinhardt domains in 
$\C^n$ ($n\geq 2$), which we call {\it horn domains}. 
These domains provide natural geometric models 
that extend the known examples 
studied in \cite{Wie84, HLT25, HaK25}, 
and they exhibit rich geometric behavior 
that has not been systematically explored before.
These domains simultaneously 
affect three fundamental problems 
in several complex variables: 
the dimensions of Bergman spaces, 
the holomorphic sectional curvature of the Bergman metric, 
and the holomorphic equivalence problem.
These three aspects are deeply interconnected. 

We now explain in more detail 
the motivation for generalizing 
the domains considered in
\cite{Wie84, HLT25, HaK25}. 
The unbounded domains studied 
in these papers have a horn-like geometry 
that narrows in one direction.
This geometric property plays a crucial role 
in the interesting results 
obtained there concerning 
the three issues mentioned above.
In the present paper, 
we consider a broader class of domains and 
show that the same phenomena essentially 
follow from an exponent that 
describes how the volume of the intersection 
of the domain with a transversally cutting 
hypersurface decreases.

This paper is organized as follows.
In Section~2, 
we give the definition of a horn domain.
In Section~3, 
we show that the Bergman spaces 
of horn domains are finite-dimensional.
In Section~4, 
we prove that the Bergman metric of 
horn domains of certain orders has constant 
holomorphic sectional curvature equal to 2.
In Section~5, 
we study holomorphic equivalence problem for
horn domains. 
\clearpage

{\it Notation and symbols.}\quad
\begin{itemize}
\item
We denote by $\N$, $\R$, $\C$ the set consisting of 
all natural numbers, real numbers, 
complex numbers, respectively. 
Moreover, we
denote by $\Z_+$, $\R_+$, $\R_{>0}$
the set consisting of all nonnegative integers, 
all nonnegative real numbers, all positive real numbers,  
respectively. 
\item
For $p:=(p_1,\ldots,p_n)\in\Z_+^n$, 
$\alpha:=(\alpha_1,\ldots,\alpha_{n-1})\in\R_{+}^{n-1}$, 
$X=(X_1,\ldots,X_n)\in\R^n$, 
$z=(z_1,\ldots,z_n)\in\C^n$, 
we denote 
\begin{eqnarray*}
&&|p|:=p_1+\cdots+p_n, \,\, 
\langle\alpha\rangle:=\alpha_1+\cdots+\alpha_{n-1}, \,\, 
z^{p}:=z_1^{p_1}\cdots z_n^{p_n}, \\
&&|X|_{\R}:=\sqrt{X_1^2+\cdots+X_n^2}, \;\;  
|z|_{\C}:=\sqrt{|z_1|^2+\cdots+|z_n|^2}.
\end{eqnarray*}
\item For $\delta>0$, we define 
${\mathbb B}(\delta):=\{z\in\C^n:|z|_{\C}<\delta\}$, 
\item
For $\epsilon>0$, we define 
\begin{equation}\label{eqn:1.1}
\begin{split}
&S^+:=\{X\in\R_+^n:|X|_{\R}=1\}, \\
&S_{\epsilon}^+
:=\{X\in S^+: X_j\geq \epsilon
\mbox{ for $j=1,\ldots,n$} \}. 
\end{split}
\end{equation}
\item
For $\epsilon>0$, we define 
\begin{equation}\label{eqn:1.2}
\begin{split}
&
H^+:=\{(X_1,\ldots,X_n)\in\R_+^n:
X_n=1\}, \\ 
&
H_{\epsilon}^+
:=\{X\in H^+: \epsilon\leq X_j \leq 1/\epsilon
\mbox{ for $j=1,\ldots,n$} \}. 
\end{split}
\end{equation}
\item 
Let $\pi:\C^n \to \R_+^n$ be the projection
defined by 
\begin{equation}
\pi(z_1\ldots,z_n)=(|z_1|, \ldots, |z_n|).
\end{equation}
\item
Let $\Omega\subset\C^n$ be a Reinhardt domain
(i.e., if $\pi(z)\in\Omega$, then $z\in\Omega$). 
We denote 
\begin{equation}\label{eqn:1.3}
\widehat{\Omega}:=\pi(\Omega)
(=\{\pi(z)\in\R_+^n: z\in\Omega\})=
\{X\in\R_+^n: X \in \Omega\}. 
\end{equation}
\item
Let $d\sigma$ be a standard surface element
on $S^+$, 
i.e., for $\xi=(\xi_1\ldots,\xi_n)\in S^+$, 
$d\sigma$ can be expressed as 
$$
d\sigma(\xi)=
\sum_{k=1}^n (-1)^k \xi_k 
d\xi_1\wedge 
\cdots 
d\xi_{k-1} \wedge 
d\xi_{k+1} \wedge 
\cdots
\wedge d\xi_n. 
$$
For an open set $U$ on $S^+$, 
we denote by $d\sigma(U)$
the volume of $U$ with respect to $d\sigma$, i.e., 
$d\sigma(U):=\int_U d\sigma$.
\item
For $U\subset \R^k$, 
we denote by ${\rm Vol}_{\R^k}(U)$
the volume of $U$ with respect to 
the Lebesgue measure on $\R^k$. 
\item 
Let $A(t)$, $B(t)$ be real-valued functions defined 
for large $t$. 
We write $A(t)\approx B(t)$ if there exist 
positive numbers $c_1, c_2, t_0$ independent of 
$t$ such that  
$c_1 B(t)\leq A(t) \leq c_2 B(t)$ for all $t\geq t_0$.
\item
For $t>0$ and $A\subset\R^n$, 
we denote $tA:=\{ta\in\R^n: a\in A\}$. 
\end{itemize}

\section{Definition of horn domains}

First, let us give a definition of 
a new class of unbounded Reinhardt domains.

\begin{definition}
Let $a>0$.  
We say that 
$\Omega\subset\C^n$ is a {\it horn domain of order} $a$ 
if $\Omega$ is a Reinhardt domain satisfying the 
following conditions 
\begin{enumerate}
\item 
There exists a positive number $\delta$ such that 
${\mathbb B}(\delta)\subset\Omega$. 
\item
There exist positive numbers $R$ and $\epsilon$ 
such that 
\begin{equation}\label{eqn:2.1}
\widehat{\Omega}\setminus\widehat{{\mathbb B}(R)}
=
\bigcup_{\rho\geq R} 
\rho\, U(\rho), 
\end{equation}
where $U(\rho)$ is an open connected subset of 
$S_{\epsilon}^+$ (see (\ref{eqn:1.1}))
and satisfies that there exist positive constants
$A_1$ and $A_2$
independent of $\rho$ such that
\begin{equation}\label{eqn:2.2}
\frac{A_1}{\rho^{2a}}\leq 
d\sigma(U(\rho))\leq 
\frac{A_2}{\rho^{2a}}
\quad 
\mbox{for $\rho\geq R$}
\quad 
\left(\Leftrightarrow 
d\sigma(U(\rho))\approx \frac{1}{\rho^{2a}}
\right).
\end{equation}
\end{enumerate}
Furthermore, we say that 
a horn domain $\Omega$ is {\it simple} 
if the sets $U(\rho)$ in (ii) satisfy
the following condition: 
\begin{enumerate}
\item[(iii)]
There exists a point $X^*\in S_{\epsilon}^+$ 
such that 
\begin{equation}
\bigcap_{\rho\geq R} U(\rho)=\{X^*\}.
\end{equation}
\end{enumerate}
\end{definition}

\begin{remark}
(1) \;
The behavior of $d\sigma(U)$ in (\ref{eqn:2.2})
can be considered in much more generality.
However, for the purposes of the problems 
discussed below,
the most essential effect is captured 
by the expression (\ref{eqn:2.2}).
We will return to this point in the final section.

\smallskip
(2) \; 
In the above definitions, 
one may drop the assumption of connectedness of
$U(\rho)$ and $\widehat{\Omega}\cap\rho S^+$, 
and instead consider a more general situation, 
for instance,
a finite union of connected open sets.
However, 
under the condition that
$U(\rho)$ is contained in $S_{\epsilon}^{+}$ 
for some $\epsilon>0$,
such a generalization merely complicates 
the arguments
without yielding any essential strengthening 
of the main results.
For this reason, 
we do not pursue this direction in the present paper.

We emphasize that the condition 
$U(\rho)\subset S_{\epsilon}^{+}$ is crucial.
Without this assumption, 
none of the results in this paper hold
(see the forthcoming paper \cite{HKN26}). 

\smallskip
(3) \;
In the definition, 
the condition (ii) can be replaced by 
the following condition.

\begin{enumerate}
\item[(ii')]
There exist positive numbers 
$R$ and $\epsilon$ such that 
\begin{equation}\label{eqn:2.3}
\widehat{\Omega}\cap
\{X\in\R_+^n: X_n\geq R\}=
\bigcup_{\rho\geq R} 
\rho V(\rho), 
\end{equation}
where $V(\rho)$ is an open and connected 
subset of $H^+_{\epsilon}$ (see (\ref{eqn:1.2}))
and satisfies that there exists 
positive numbers $A_1, A_2$ such that
\begin{equation}\label{eqn:2.4}
\frac{A_1}{\rho^{2a}}\leq 
{\rm Vol}_{\R^{n-1}}(V(\rho))\leq 
\frac{A_2}{\rho^{2a}}
\quad \mbox{for $\rho\geq R$}
\quad 
\left(\Leftrightarrow 
{\rm Vol}_{\R^{n-1}}(V(\rho))\approx \frac{1}{\rho^{2a}}
\right).
\end{equation}
\end{enumerate}
\end{remark}

To demonstrate that 
the above class provides 
a substantial generalization of 
the known specific unbounded domains 
in \cite{Wie84, HLT25, HaK25}, 
we introduce a special class of horn domains, 
defined in a more explicit manner.
For $\alpha=(\alpha_1\ldots,\alpha_{n-1})
\in\R_{>0}^{n-1}$, 
let 
${\mathcal H}_n(\alpha)$ denote the set of 
Reinhardt domains 
$\Omega\subset\C^n$ 
satisfying the following conditions:
\begin{enumerate}
\item
For each $\rho>0$, 
$\widehat{\Omega}\cap \rho S^+$ is an open and 
connected subset of $\rho S^+$. 
\item 
There exist positive constants $c_1, c_2$ such that 
\begin{equation}\label{eqn:2.5}
c_1 D(\alpha) 
\subset \Omega \subset
c_2 D(\alpha), 
\end{equation}
where $D(\alpha)$ is the domain defined by 
\begin{equation}\label{eqn:2.6}
D(\alpha)=
\left\{
z\in\C^n: 
\left||z_k|^2-|z_n|^2\right|<
\frac{1}{(|z_n|^2+1)^{\alpha_k}}
\text{ for } k=1,\ldots,n-1
\right\}.
\end{equation}
%
\end{enumerate}

We note
that the construction of the above $D(\alpha)$ 
is inspired by the following two known 
examples in $\C^2$:
\begin{equation}\label{eqn:2.7}
\begin{split}
&D_{\rm Wie}(\alpha):=
\left\{z\in\C^2:  
\big||z_1|-|z_2|\big|\cdot 
(|z_1|+|z_2|)^{2\alpha}<1
\right\};\\
&D_{\rm HL}(\alpha):=
\left\{z\in\C^2: 
\left||z_1|^2-|z_2|^2\right|<\frac{1}{(1+|z_1|^2+|z_2|^2)^{\alpha}}
\right\},
\end{split}
\end{equation}
where $\alpha\in\R_{>0}$. 
The above two examples are 
given in the papers by Wiegerinck \cite{Wie84} and
by Huang, Li and Treuer \cite{HLT25}, respectively. 
(The domain $D_{\rm Wie}(\alpha)$ is modified 
from our perspective.) 
It follows from the implicit function theorem that 
$D_{\rm Wie}(\alpha)$ and $D_{\rm HL}(\alpha)$ 
belong to the classes ${\mathcal H}_2(\alpha)$. 
We remark that 
the boundary of $D(\alpha)$ is not smooth 
when $n\geq 3$, but it is easy to 
construct domains belonging to 
the class ${\mathcal H}_n(\alpha)$ 
with smooth boundaries.

\begin{lemma}
Let $\alpha\in\R_{>0}^{n-1}$. 
If $\Omega\in{\mathcal H}_n(\alpha)$, then
$\Omega$  
is a simple horn domain of order 
$\langle\alpha\rangle+n-1$.
\end{lemma}
\begin{proof}
In order to prove the above lemma, 
it suffices to show that $D(\alpha)$ 
is a simple horn domain of order $\langle\alpha\rangle+n-1$.

First, 
the existence of a positive number $\delta$ 
in (i) and 
the simplicity condition in (iii) are easily verified. 
Next, setting
\begin{equation}\label{eqn:2.8}
V(\rho)=
\frac{1}{\rho}(\widehat{D(\alpha)}\cap \rho H^+), 
\end{equation}
we can obtain the equation in (\ref{eqn:2.3}), 
i.e.,  
$\widehat{D(\alpha)}\cap 
\{X\in\R_+^n:X_{n}\geq R\}=
\bigcup_{\rho\geq R}\rho V(\rho)$.
Therefore, in order to complete the proof of the lemma,
it suffices to show that
there exists $\epsilon>0$ such that
$V(\rho)\subset H_{\epsilon}^+$ 
for $\rho\geq 1$
and that
${\rm Vol}_{\R^{n-1}}(V(\rho))\approx 
\frac{1}{\rho^{2(\langle\alpha\rangle+n-1)}}$
for $\rho>1$.

From the definition of $D(\alpha)$ in 
(\ref{eqn:2.6}), 
\begin{equation*}
\begin{split}
&\quad
\widehat{D(\alpha)}\cap \rho H_+= 
\Biggl\{
(X_1,\ldots,X_{n-1},\rho) \in\R_+^n:  \biggr. \\
& \left. \quad\quad\quad
\sqrt{\rho^2-\frac{1}{(\rho^2+1)^{\alpha_k}}} 
<X_k
<\sqrt{\rho^2+\frac{1}{(\rho^2+1)^{\alpha_k}}} 
\mbox{ for $k=1,\ldots, n-1$}
\right\}.
\end{split}
\end{equation*}
Setting 
\begin{equation}\label{eqn:2.10}
\delta_k^{(\pm)}(\rho):=\frac{1}{\rho}
\sqrt{\rho^2\pm \frac{1}{(\rho^2+1)^{\alpha_k}}}-1, 
\end{equation}
respectively, we can see  
\begin{equation}\label{eqn:2.11}
\begin{split}
&
V(\rho)=
\frac{1}{\rho}(\widehat{D(\alpha)}\cap \rho H_+)
= \\
&\{
(X_1,\ldots,X_{n-1},1)\in \R_+^n:
\delta_k^{(-)}(\rho)<X_k-1<\delta_k^{(+)}(\rho)
\mbox{ for $k=1,\ldots, n-1$}\}. 
\end{split}
\end{equation}
From the above expression of $V(\rho)$, 
there exists $\epsilon>0$ such that 
$V(\rho)\subset H_{\epsilon}^+$ for $\rho\geq 1$. 
Note that Taylor's formula implies 
\begin{equation}\label{eqn:2.12}
|\delta_k^{(\pm)}(\rho)|
\approx 
\frac{1}{\rho^{2\alpha_k+2}} 
\quad \mbox{ for $\rho\geq 1$.}
\end{equation}
From (\ref{eqn:2.11}), (\ref{eqn:2.12}),  
we see 
\begin{equation*}
{\rm Vol}_{\R^{n-1}}(V(\rho))
\approx 
\prod_{k=1}^{n-1} |\delta_k^{(\pm)}(\rho)|
\approx 
\prod_{k=1}^{n-1} 
\frac{1}{\rho^{2\alpha_k+2}}
=\frac{1}{\rho^{2(\langle\alpha\rangle+n-1)}}
\; \mbox{ for $\rho\geq 1$,}
\end{equation*}
which completes the proof of the lemma.
\end{proof}


\section{Dimensions of the Bergman spaces}

Let $\Omega$ be a domain in $\C^n$ and 
let $A^2(\Omega)$ be the Hilbert space 
of the $L^2$-holomorphic functions on $\Omega$, i.e., 
$$
A^2(\Omega)=
\left\{
f:\Omega\to\C \mbox{ holomorphic}:
\|f\|_{\Omega}:=
\left(
\int_{\Omega} |f(z)|^2 dV(z)
\right)^{1/2}<\infty
\right\},
$$
where $dV$ is the Lebesgue measure on $\C^n$.

When $\Omega$ is a bounded domain, 
$A^2(\Omega)$ contains all the monomials, 
which implies that 
the dimension of $A^2(\Omega)$ is infinity. 
On the other hand,  
the issue of the dimension of 
$A^2(\Omega)$ is not obvious in the unbounded case. 
When $\Omega$ is a domain in the complex plane $\C$, 
it is seen in \cite{Wie84} that 
\begin{equation}\label{eqn:3.1}
\text{$A^2(\Omega)$
is either trivial, i.e., $\{0\}$,  or infinite dimensional.}
\end{equation}
Surprisingly, in the same paper \cite{Wie84}, 
Wiegerinck showed that if $\alpha>1$, then
$D_{\rm Wie}(\alpha)$ in (\ref{eqn:2.7}) 
with $\alpha\in\R_{>0}$ 
has a finite-dimensional 
and nontrivial Bergman space.
More recently, 
Huang, Li and Treuer \cite{HLT25} also showed that 
the domain $D_{\rm HL}(\alpha)$ in (\ref{eqn:2.7}) 
has the same properties. 

The following theorem shows that 
the above results can be widely generalized.

\begin{theorem}
Let $\Omega\subset\C^n$ be a horn 
domain of order $a$. 
If $a>n$, then 
the Bergman space of $A^2(\Omega)$ is given by
\begin{equation}\label{eqn:3.2}
A^2(\Omega)={\rm Span}\left\{z^p:
p\in\Z_+^n, |p|<a-n
\right\}.
\end{equation}
More precisely, 
the dimension of $A^2(\Omega)$ 
can be expressed as 
\begin{equation}\label{eqn:3.3}
\dim A^2(\Omega)=
\begin{cases} 
&0 \quad  \mbox{ if  $ 0<a\leq n$,} \\
&\displaystyle \binom{n+k}{n}  \quad
\mbox{if  
$n+k<a\leq n+k+1$ 
 $(k\in\Z_+)$}.
\end{cases}
\end{equation}
\end{theorem}

\begin{proof}
Denoting
$\pi(z)=X=(X_1,\ldots,X_n)\in\R_+^n$ 
for $z\in\C^n$,  
we have
\begin{equation*}
\int_{\Omega\setminus {\mathbb B}(R)} 
|z_1|^{2p_1} \cdots |z_n|^{2p_n}dV(z)
=
(2\pi)^n 
\int_{\widehat{\Omega}\setminus
\widehat{{\mathbb B}(R)}} 
X_1^{2 p_1+1}\cdots X_n^{2 p_n+1} 
dV(X).
\end{equation*}
Let $\rho>0$ and 
$\xi=(\xi_1,\ldots,\xi_n)\in S^+$.
When we express 
$X_j=\rho \, \xi_j$ for $j=1,\ldots,n$, 
$dV(X)=\rho^{n-1} d\sigma(\xi)\wedge d\rho$ holds.
Therefore, Fubini's theorem implies 
\begin{equation*}
\int_{\Omega\setminus {\mathbb B}(R)} 
|z_1|^{2p_1} \cdots |z_n|^{2p_n}dV(z)
=(2\pi)^n \int_1^{\infty} K(\rho) 
\rho^{2|p|+2n-1}d\rho, 
\end{equation*}
where 
$$
K(\rho)=\int_{U(\rho)} 
\xi_1^{2p_1+1}\cdots
\xi_n^{2p_n+1} d\sigma(\xi).
$$
Since $U(\rho)\subset S_{\varepsilon}^+$,
$
\xi_1^{2p_1}
\cdots
\xi_n^{2p_n} \approx 1$ for $\xi\in U(\rho)$, 
which implies 
\begin{equation*}
K(\rho)\approx
d\sigma(U(\rho))
\approx \frac{1}{\rho^{2a}}.
\end{equation*}
Therefore, 
we have 
\begin{equation}\label{eqn:3.4}
\begin{split}
&\|z^{2p}\|_{\Omega} <\infty
\Longleftrightarrow 
\int_R^{\infty} 
\rho^{2|p|+2n-1-2a}d\rho <\infty \\
&\quad \Longleftrightarrow
2|p|+2n-1-2a<-1
\Longleftrightarrow
|p|<a-n,
\end{split}
\end{equation}
which implies 
$A^2(\Omega)\supset 
{\rm Span}\{z^p:
p\in\Z_+^n, |p|< a-n
\}$.

Conversely, the opposite inclusion can be 
established in a similar manner 
as in the proof of the main theorem 
in \cite{Wie84}.

We remark that 
since $|p|$ is nonnegative, 
the value of $a$ must be greater than 
$n$ from (\ref{eqn:3.1}). 

\end{proof}

\begin{remark}
The pseudoconvexity of a domain 
has a significant influence on 
the dimension of its Bergman space. 
Wiegerinck \cite{Wie84} conjectured that 
property (\ref{eqn:2.1}) always holds 
for every pseudoconvex domain $\Omega$.
As far as we know, this conjecture remains 
unsolved at present, 
but it has already been affirmatively 
settled in many cases
(\cite{Zwo99, Zwo00, Eng07, Juc12, 
GHH17, PfZ17}).
In particular, 
it has been confirmed 
in the case of Reinhardt domains
(\cite{Zwo99, Zwo00, Eng07}). 
From these results,  
if the order is greater than $n$, 
then the respective horn domain
cannot be pseudoconvex. 
\end{remark}

\section{Holomorphic sectional curvatures 
of the Bergman metrics}

Let $\Omega$ be a domain in $\C^n$. 
The {\it Bergman kernel} $K_{\Omega}(z,w)$ of $\Omega$ 
is defined by 
$$
K_{\Omega}(z,w)=\sum_{\alpha} 
\phi_{\alpha}(z) \overline{\phi_{\alpha}(w)},
$$ 
where $\{\phi_{\alpha}\}_{\alpha}$ is a complete orthonormal 
basis of $A^2(\Omega)$.  
Its restriction to the diagonal, $K_{\Omega}(z,z)$,  
is denoted simply by $K_{\Omega}(z)$. 

It has recently been recognized 
that the issue of the dimension of $A^2(\Omega)$ is
closely related to that of 
constant holomorphic sectional curvatures
of the Bergman metric. 
Recently, 
Huang, Li and Treuer \cite{HLT25} has obtained 
several significant results on this topic.
In their paper, they show that the Bergman metric 
of the domain $D_{\rm HL}(\alpha)$ 
in (\ref{eqn:1.2}) has constant holomorphic 
sectional curvature equal to $2$. 
Their results can be widely generalized as follows.

\begin{theorem}
Let $\Omega\subset\C^n$ be a horn domain 
of order $a$.
If 
$n+1<a\leq n+2$,
then 
the Bergman metric of $\Omega$ has constant 
holomorphic sectional curvature equal to $2$.
\end{theorem}

\begin{proof}
It follows from Theorem~3.1 that 
the condition 
$n+1<a\leq n+2$ is equivalent to 
$A^2(\Omega)={\rm Span}\{1,z_1,\ldots,z_n\}$. 
Therefore, 
the Bergman kernel of $\Omega$ can be written
as 
\begin{equation}\label{eqn:4.1}
K_{\Omega}(z)=c_0+\sum_{k=1}^n c_k |z_k|^2
\quad \mbox{ for $z\in \Omega$}, 
\end{equation}
where 
$$
c_0={\rm Vol} (\Omega)^{-1}, 
\quad 
c_k=\|z_k\|_{\Omega}^{-2}=
\left(\int_{\Omega} |z_k|^2 dV(z)\right)^{-1}.$$ 
Furthermore, we have
$$
K_{\Omega}(z)=
c_0\left(1+\sum_{k=1}^n \tilde{c}_k |z_k|^2\right)
\quad \mbox{ for $z\in \Omega$},
$$ 
where $\tilde{c}_k=c_k/c_0$ for
$k=1,\ldots,n$. 
Through
the scaling $(w_1,\ldots,w_n)=
(\sqrt{\tilde{c}_1}z_1,\ldots,
\sqrt{\tilde{c}_n}z_n)$,
the Bergman metric of $\Omega$ 
is isometric to the Fubini-Study metric
of $\PP^n$ restricted to 
a domain in the cell $\C^n$ with 
$$
\omega_{FS} =
\partial\bar{\partial} \log 
\left(
1+\sum_{k=1}^n |w_k|^2
\right). 
$$
Therefore, the Bergman metric of $\Omega$ 
has a constant holomorphic 
sectional curvature equal 
to $2$.
\end{proof}


\section{Holomorphic equivalence problem for horn domains}

A. Yamamori \cite{Yam25} obtained the following result 
on holomorphic equivalence problem 
for Reinhardt domains 
with very small Bergman spaces.

\begin{theorem}[\cite{Yam25}, see also \cite{HaK25}]
Let $\Omega_1, \Omega_2\subset\C^n$ be 
Reinhardt domains such that $0\in\Omega_j$ and
$A^2(\Omega_j)={\rm Span}\{1,z_1, \ldots,z_n\}$ for
$j=1,2$. 
If $\Omega_1$ and $\Omega_2$ are biholomorphic, 
then 
every biholomorphism $f:\Omega_1\to\Omega_2$ 
is linear. 
\end{theorem}

In the case of simple horn domains, 
in addition to the linearity of biholomorphisms, 
their structure is completely 
determined up to permutation 
and coordinate-wise scaling:

\begin{proposition}
Let $\Omega_1, \Omega_2\subset\C^n$ be
simple horn domains 
that are biholomorphic to each other.   
If a biholomorphism $F : \Omega_1\to\Omega_2$ 
is linear, then
$F$ can be written in the form
\begin{equation}\label{eqn:5.1}
F(z) = 
\bigl(
C_1 z_{\sigma(1)}, \dots, C_n z_{\sigma(n)}
\bigr),
\end{equation}
where each $C_j$ is a nonzero complex number and 
$\sigma$ is a permutation of $\{1,\dots,n\}$ 
(i.e., $\sigma \in S_n$).
\end{proposition}
\begin{proof}
For $X=(X_1, \ldots, X_n)\in\R_{>0}^n$, 
we define
\begin{equation*}
\begin{split}
&L_{\R}(X):=
\{t X=(t X_1, \ldots, t X_n)\in\R_+^n: t\geq 0\}, \\
&L_{\C}(X):=
\{\zeta X=(\zeta X_1, \ldots, \zeta X_n) \in \C^n:
\zeta\in \C\}.
\end{split}
\end{equation*}
For $\theta=(\theta_1,\ldots,\theta_n)\in\R^n$, 
let $\sigma_{\theta}:\C^n\to\C^n$ be the map defined by 
\begin{equation*}
\sigma_{\theta}(z)=
(e^{i\theta_1} z_1,\ldots, e^{i\theta_n} z_n). 
\end{equation*}

For $j=1,2$, the domain
$\Omega_j$ satisfies the conditions 
(i), (ii) in Definition~2.1   
with $R=R_j$, $U(\rho)=U_j(\rho)$, 
respectively. 
Moreover, $\Omega_1$ satisfies 
the simplicity condition (iii),  that is, 
there exists a point $X^*\in S_{\epsilon}^+$ 
such that 
\begin{equation*}
\bigcap_{\rho\geq R_1} U_1(\rho)
=\{X^*\}=\{(X_1^*,\ldots,X_n^*)\}. 
\end{equation*}
We then observe that 
\begin{equation}\label{eqn:5.2}
L_{\R}(X^*)\subset \widehat{\Omega}_1, \quad
L_{\C}(X^*)\subset \Omega_1. 
\end{equation}
We write $w=F(z)$ in coordinates as 
\begin{equation}\label{eqn:5.3}
w_j=\sum_{k=1}^n C_{jk} z_k \quad 
\mbox{ for } j=1,\ldots, n, 
\end{equation}
where each $C_{jk}$ is a complex number. 
Let $Y^*=(Y_1^*,\ldots,Y_n^*)\in\R_{>0}^n$ 
be the point defined by 
\begin{equation*}
Y_j^*=\sum_{k=1}^n |C_{jk}| X_k^* 
\quad \mbox{for $j=1,\ldots,n$}.
\end{equation*}
The following lemma 
plays a crucial role 
in the proof of the proposition. 
\begin{lemma}
There exist $\theta^*, \varphi^*\in [0,2\pi)^n$ 
such that 
\begin{equation}\label{eqn:5.4}
(F\circ\sigma_{\theta^*})(X^*)=\sigma_{\varphi^*}(Y^*). 
\end{equation}
\end{lemma}
\begin{proof}
From (\ref{eqn:5.3}), we have
\begin{equation*} 
(F\circ\sigma_{\theta})(X^*) =
\left(
\sum_{k=1}^n C_{1k} e^{i\theta_k} X_k^*, 
\ldots, 
\sum_{k=1}^n C_{nk} e^{i\theta_k} X_k^*
\right).
\end{equation*} 
Moreover, for each $j=1,\ldots, n$, 
we obtain
\begin{equation*}
\left|
\sum_{k=1}^n C_{jk} e^{i\theta_k} X_k^*
\right|
\leq 
\sum_{k=1}^n |C_{jk}| X_k^*=Y_j^*.
\end{equation*}
Therefore, for any $\theta\in [0,2\pi)^n$ and
$j=1,\ldots,n$, 
there exist $\rho_j(\theta)\in(0,1]$ and 
$\varphi_j(\theta)\in[0,2\pi)$ such that 
\begin{equation}\label{eqn:5.5}
\sum_{k=1}^n C_{jk} e^{i\theta_k} X_k^*=
\rho_j(\theta) e^{i \varphi_j(\theta)} Y_j^*.
\end{equation}
Since $\sigma_{\theta}(tX^*) 
\in\Omega_1$ for all 
$t>0$, $\theta\in[0,2\pi)^n$ 
from (\ref{eqn:5.2}), 
we obtain
\begin{equation*}
\bigl(
t\rho_1(\theta) e^{i\varphi_1(\theta)}Y_1^*, 
\ldots, 
t\rho_n(\theta) e^{i\varphi_n(\theta)}Y_n^*
\bigr)
\in \Omega_2 \quad 
\text{ for all } t>0, \theta\in[0,2\pi)^n.
\end{equation*}

For $j=1,\ldots, n$, define 
$\theta^j=(\theta^j_1,\ldots,\theta^j_n)
\in
[0,2\pi)^n$ 
by 
\begin{equation}\label{eqn:5.6}
\theta^j_{k}:=
\begin{cases}
&\!\!\!\!\! -\arg(C_{jk}), 
\quad \text{ if } C_{jk}\neq 0, \\
& 0, \quad \text{ if } C_{jk}=0.
\end{cases}
\end{equation}
Then $\rho_j$ and $\varphi_j$ in (\ref{eqn:5.5})
satisfy 
$\rho_j(\theta^j)=1$ and $\varphi_j(\theta^j)=0$, 
that is,  
\begin{equation}\label{eqn:5.7}
\sum_{k=1}^n C_{jk} e^{i\theta^j_{k}} X_k^*
=
Y_j^* \quad 
\text{\; for } j=1,\ldots, n.
\end{equation}
For $l=1,\ldots,n$, set
\begin{equation*}
Y^l=(
\rho_1(\theta^l)Y_1^*,\ldots, 
\rho_n(\theta^l)Y_n^*).
\end{equation*}
Since $L_{\R}(Y^l)\subset \widehat{\Omega}_2$ 
for all $l$, 
the simplicity condition of $\Omega_2$ 
implies that 
$Y^l=a_{lk} Y^k$ with positive constants $a_{lk}$ 
for all $k, l$. 
Moreover, using $\rho_j(\theta^j)=1$ for all $j$, 
we conclude that 
$$
\rho_k(\theta^l)=1 \quad 
\text{ for all } k, l, 
$$
which is equivalent to 
$Y^1=\cdots=Y^n=Y^*$. 
Therefore, setting 
$\theta^*=\theta^1$ and 
$\varphi^*=(\varphi_1(\theta^1),\ldots,
\varphi_n(\theta^1))$ 
yields the assertion. 
\end{proof}

Since
$\sigma_{\theta}(\zeta X^*)\in \Omega_1$ 
for all $\theta\in[0,2\pi)^n$, 
$\zeta\in\C$ from (\ref{eqn:5.2}), 
Lemma~5.3 implies that 
$$
\zeta Y^*
=\zeta (\sigma_{\varphi^*}^{-1}
\circ F \circ \sigma_{\theta^*})(X^*) =
\sigma_{\varphi^*}^{-1}
(F(\sigma_{\theta^*}(\zeta X^*)))
\in\Omega_2 \quad 
\mbox{ for all } \zeta\in\C. 
$$
Consequently, 
\begin{equation}\label{eqn:5.8}
\bigcap_{\rho>R_2} U_2(\rho)=
\left\{\frac{Y^*}{|Y^*|}\right\}, 
\quad
L_{\R}(Y^*)\subset \widehat{\Omega}_2, 
\quad 
L_{\C}(Y^*)\subset {\Omega}_2. 
\end{equation}
We define the biholomorphism
$F^*:\Omega_1\to\Omega_2$ by 
\begin{equation*}
F^*:=
\sigma_{\varphi^*}\circ 
F \circ \sigma_{\theta^*}^{-1},
\end{equation*}
so that the following diagram commutes:
\begin{equation*}
\begin{tikzcd}
\Omega_1 \arrow[r,"F"] \arrow[d,"\sigma_{\theta^*}"'] 
& \Omega_2 \arrow[d,"\sigma_{\varphi^*}"] \\
\Omega_1 \arrow[r,"F^*"] 
& \Omega_2
\arrow[phantom, from=1-1, to=2-2, "\circlearrowright"]
\end{tikzcd}
\end{equation*}
It follows from (\ref{eqn:5.4}) in Lemma~5.3 
that  
$Y^*=F^*(X^*)$, 
which implies that 
$w=F^*(z)$ can be written in coordinates as 
\begin{equation*}
w_j=\sum_{k=1}^n |C_{jk}| z_k \quad 
\mbox{\; for } j=1,\ldots, n.
\end{equation*}
We remark that 
$F^*$ can be extended to 
a biholomorphism from 
$\C^n$ to $\C^n$ and 
that 
the restriction of $F^*$ 
to $\R^n$ is bijective from 
$\R^n$ to $\R^n$. 

On the other hand, 
by considering similarly 
linear biholomorphisms from 
$\Omega_2$ to $\Omega_1$, 
we obtain analogous results to 
those of Lemma~5.3.
Let $G^*$ denote the inverse of $F^*$. 
Arguing as in the case of $F^*$, 
we can see 
that $z=G^*(w)$ can be written 
in coordinates as  
\begin{equation*}
z_j=\sum_{k=1}^n D_{jk} w_k 
\quad \mbox{ for $j=1,\ldots,n$},
\end{equation*}
where each $D_{jk}$ is a nonnegative number. 
The equation $F^*\circ G^*={\rm Id}$ implies that 
\begin{equation*}
\sum_{l=1}^n |C_{jl}| D_{lk} =\delta_{jk}, 
\end{equation*}
where $\delta_{jk}=1$ if $j=k$ or $=0$ if $j\neq k$. 
Therefore, 
for each $j$, there exists a unique $k$ such that 
$C_{jk}\neq 0$ and $C_{jl}=0$ for all $l\neq k$.
Since the $n\times n$ matrix $(|C_{jk}|)_{jk}$ 
is invertible, 
the correspondence $j\mapsto k$ defines 
a permutation $\sigma\in S_n$. 
Consequently, $F$ can be written in the form
\[
F(z)
=
\bigl(
C_1 z_{\sigma(1)}, \ldots, C_n z_{\sigma(n)}
\bigr),
\]
where each $C_j \neq 0$ and
$\sigma \in S_n$.
This completes the proof of the proposition.
\end{proof}



From Theorem~5.1 and Proposition~5.2,  
we can obtain the following result.

\begin{theorem}
Let $n+1 < a_1, a_2 \leq n+2$. 
Let $\Omega_1, \Omega_2 \subset \C^n$ be 
simple horn domains of orders $a_1$, $a_2$, 
respectively.
If $\Omega_1$ and $\Omega_2$ are 
biholomorphic, then
every biholomorphism $F:\Omega_1\to\Omega_2$ 
can be written as in the form (\ref{eqn:5.1}) and 
$a_1 = a_2$.
\end{theorem}
\begin{proof}
Let $F:\Omega_1\to\Omega_2$ 
be a biholomorphism. 
Applying Theorem~3.1 to Theorem~5.1, 
we see that $F$ is linear. 
Moreover, Proposition~5.2 implies that
$F$ can be written as in the form (\ref{eqn:5.1}).
Without loss of generality,
we may assume that the permutation 
in Proposition~5.2 satisfies 
$\sigma(j)=j$ for $j=1,\ldots, n$.

Let $c_j:=|C_j|$ for $j=1,\ldots,n$, where 
each $C_j$ is as in (\ref{eqn:5.1}).  
Define a map $\hat{F}: \R^n_+\to \R_+^n$ by
$$\hat{F}(X_1\ldots, X_n):=(c_1 X_1,\ldots, c_n X_n). $$
This induced map 
$\hat{F}:\widehat{\Omega}_1\to \widehat{\Omega}_2$ 
is bijective.   
For positive numbers $t$ and $s$,  set 
\begin{equation*}
V(t)=\frac{1}{t}(\widehat{\Omega}_1 \cap t H_+), 
\qquad
W(s)=\frac{1}{s}
(\widehat{\Omega}_2 \cap s H_+), 
\end{equation*}
where $H_+$ is as in (\ref{eqn:1.2}).
When $s=c_n t$, we have
$$
\hat{F}(t V(t))=s W(s).
$$
Therefore, 
\begin{equation*}
\begin{split}
&{\rm Vol}_{\R^{n-1}}(W(s))=
{\rm Vol}_{\R^{n-1}}\left(\frac{1}{s}\hat{F}(tV(t))\right)=
{\rm Vol}_{\R^{n-1}}\left(\frac{1}{c_n t} t\hat{F}(V(t))\right) \\
&\quad\quad
=\frac{1}{c_n}{\rm Vol}_{\R^{n-1}}(\hat{F}(V(t))) 
=\frac{1}{c_n}\prod_{j=1}^{n-1} c_j \cdot
{\rm Vol}_{\R^{n-1}}(V(s/c_n)).
\end{split}
\end{equation*}
Since the order of $\Omega_1$ is $a_1$,   
the quantity ${\rm Vol}_{\R^{n-1}}(V(t))$ is estimated 
as in (\ref{eqn:2.4}) with $a=a_1$. 
The above equalities then imply  that, 
for sufficiently large $s$, 
\begin{equation*}
\frac{B_1}{s^{2a_1}}
\leq 
{\rm Vol}_{\R^{n-1}}(W(s))
\leq 
\frac{B_2}{s^{2a_1}}, 
\end{equation*}
where 
$$
B_j=A_j \cdot 
c_n^{2a_1-1} \prod_{k=1}^{n-1} c_k \;\;
\mbox{ for $j=1,2$.}
$$ 
Hence the order of $\Omega_2$ is also $a_1$. 
Thus, $a_1=a_2$. 
\end{proof}

From Lemma~2.3, Theorem~5.5 implies
the following result.

\begin{corollary}
Let $\alpha,\alpha'\in\R_{>0}^{n-1}$ satisfy
$2<\langle\alpha \rangle, \langle\alpha'\rangle \leq 3$. 
Let $\Omega\in{\mathcal H}_n(\alpha)$ and 
$\Omega' \in{\mathcal H}_n(\alpha')$. 
If $\Omega$ and $\Omega' $ are 
biholomorphic, then every biholomorphism 
$F:\Omega_1\to\Omega_2$ can be written 
as in the form (\ref{eqn:5.1})
and $\langle\alpha\rangle=\langle\alpha'\rangle$.
\end{corollary}
In particular, the above corollary 
implies that, when $2<\alpha, \alpha'\leq 3$, 
$D_{\rm HL}(\alpha)$ and $D_{\rm HL}(\alpha')$ 
are biholomorphic, then $\alpha=\alpha'$, 
which is shown in \cite{Yam25}. 
Of course, in the cases of 
$D_{\rm Wie}(\alpha)$ and $D(\alpha)$ in Section~2, 
the same results can also be obtained.  


\section{Remarks on the definition of horn domains}

In the definition of horn domains in Section~2,
the essential requirement is 
to describe the behavior of $d\sigma(U(\rho))$ 
as $\rho\to\infty$. 
Let us recall the condition 
on $d\sigma(U(\rho))$:
there exist positive constants
$A_1$, $A_2$ such that
\begin{equation}\label{eqn:6.1}
\frac{A_1}{\rho^{2a}}\leq 
d\sigma(U(\rho))\leq 
\frac{A_2}{\rho^{2a}}
\quad 
\mbox{for $\rho\geq R$.}
\end{equation}
This behavior of 
$d\sigma(U(\rho))$ 
may be treated in a variety of broader settings.
Here we consider the following two cases.

\smallskip 
(1) \;\;
The following condition is a 
slight generalization of (\ref{eqn:6.1}).
Let $a\in\R_{>0}$ and $b\in\R$. 
Consider the condition: 
there exist positive constants 
$A_1$ and $A_2$ 
such that
\begin{equation}\label{eqn:6.2}
\frac{A_1}{\rho^{2a} (\log \rho)^b}
\leq
d\sigma(U(\rho))
\leq
\frac{A_2}{\rho^{2a} (\log \rho)^b}
\quad 
\mbox{for $\rho\geq R$.}
\end{equation}
If the condition~(\ref{eqn:6.1}) 
in the definition of horn domains 
is replaced by the condition~(\ref{eqn:6.2}), 
then the Bergman space $A^2(\Omega)$ 
of the corresponding horn domain 
$\Omega$ satisfies the following:
\begin{enumerate}
\item
If $b\leq 1$, 
then the conclusion of Theorem~3.1 remains valid.

\item
If $b>1$ and $a\geq n$, then
\begin{equation*}
A^2(\Omega)=
{\rm Span}
\{z^p : p\in\mathbb{Z}_+^n,\ |p|\le a-n \}.
\end{equation*}
More precisely, the dimension of $A^2(\Omega)$
is given by
\begin{equation*}
\dim A^2(\Omega)=
\begin{cases} 
&0 \quad  \mbox{ if  $ 0<a<n$,} \\
&\displaystyle \binom{n+k}{n}  \quad
\mbox{if  
$n+k\leq a < n+k+1$ 
 $(k\in\Z_+)$}.
\end{cases}
\end{equation*}
\end{enumerate}
Therefore,  in the case
where $\Omega$ is the corresponding horn domain, 
if $b>1$, then in Theorems~4.1 and 5.3 the assumption
$n+1<a, a_1, a_2 \leq n+2$ must 
be replaced by 
$n+1\leq a, a_1, a_2 < n+2$.
\smallskip

(2) \;\;
Next, we consider an extremely narrow 
horn-shaped situation.
Consider the condition:
there exists a positive constant 
$A$  such that
\begin{equation}\label{eqn:6.3}
d\sigma(U(\rho))
\leq
A e^{-\rho}
\quad 
\mbox{for $\rho\geq R$.}
\end{equation}
If the condition~(\ref{eqn:6.1}) is 
replaced by the condition~(\ref{eqn:6.3}) 
in the definition of horn domains, 
then the Bergman space of 
the corresponding horn domain
becomes infinite-dimensional.

Since the above statements in (1) and (2) 
can be proved 
in essentially 
the same manner as Theorem~3.1, 
we leave the details to the reader.

\medskip


{\it Acknowledgements.} 
The authors are grateful to Atsushi Yamamori for pointing out a serious error in an earlier draft of the paper.
Joe Kamimoto was supported by JSPS KAKENHI Grant Numbers JP26K06847.
Ninh Van Thu was supported by the Vietnam National Foundation for Science and Technology Development (NAFOSTED) under grant number 101.02-2021.42.

\section*{Conflict of Interest}
The authors declare that they have no conflict of interest.

\section*{Data Availability}
No data were generated or analyzed in this study.
Data sharing is therefore not applicable to this article.

\end{document}